\documentclass[10pt,reqno]{amsart}
\usepackage{amsfonts,amssymb,amscd,amsmath,enumerate,verbatim,calc}
\usepackage{color,soul}
\usepackage{mathtools}
\usepackage{graphicx}
\usepackage{ragged2e}

\usepackage{setspace}
\usepackage[margin=1in,left=1.2in]{geometry}
\usepackage[charter]{mathdesign}

\everymath{\displaystyle}
\usepackage{fancyhdr}

\newcommand{\N}{\mathbb{N} }

\newcommand{\C}{\mathbb{C} }

\theoremstyle{plain}

\newtheorem{theorem}{Theorem}[section]
\newtheorem{corollary}[theorem]{Corollary}

\theoremstyle{definition}
\newtheorem{definition}[theorem]{Definition}

\newtheorem{remark}[theorem]{Remark}
\newtheorem{example}[theorem]{Example}

\vspace{-\baselineskip}%

\begin{document}
\title{RANK OF BICOMPLEX MATRICES AND SYSTEM OF ALGEBRAIC EQUATIONS}
	\author{Amita}
	\email{amitasharma234@gmail.com}
	\address{Department of Mathematics, Indira Gandhi National Tribal University, Amarkantak, Madhya Pradesh 484886, India}

	\author{Akhil~Prakash}
	\email{akhil.sharma140@gmail.com}
	\address{Department of Mathematics, Aligarh Muslim University, Aligarh, Uttar Pradesh 202002, India}

	\author{Mamta Amol Wagh}
	\address{Department of Mathematics, University of Delhi, New Delhi 110078, India}
	\email{mamtanigam@ddu.du.ac.in}

	\author{Suman Kumar}
	\address{Department of Mathematics, Indira Gandhi National Tribal University, Amarkantak, Madhya Pradesh 484886, India}
	\email{suman@igntu.ac.in}
	

	\subjclass[IMS] {Primary 15A03, 15A04, 15A24, 15A30; Secondary 30G35}
	\date{\today}
	
	\begin{abstract}
	In this paper, we study the rank of matrices of bicomplex numbers. The relationship between rank, idempotent column rank and idempotent row rank is examined. Then, the solution of a system of equations in bicomplex space is presented using a new technique. Moreover, we establish a necessary and sufficient condition for the existence of solutions of a system of equations in bicomplex space and derive some related results.
		\end{abstract}
	\maketitle
	
	\section*{Introduction} 
	The theory of bicomplex numbers represents a captivating area of mathematical inquiry, offering a deeper understanding of number systems and their applications across different disciplines. The book of Price \cite{price2018}  offers an unparalleled and exhaustive examination of bicomplex numbers, while Segre [2] laid the groundwork for subsequent investigations into the properties and applications of bicomplex numbers. For the recent research in this area, we refer the reader to (\cite{alpay2014basics,  futagawa1928, futagawa1932, riley1953, ringleb1933,  rochon2004, srivastava2003, srivastava2008, anjali2025}).

\vspace{1.5mm}

This paper utilizes the following approach: in section 1, we summarize some basics of bicomplex numbers, bicomplex linear transformation, and some necessary results which are used in this research. Section 2 deals with the rank of a bicomplex matrix and some related results. Also, it focuses on finding the solution of a system of bicomplex linear equations.

\section{Preliminaries and Notations}
	\begin{definition}\label{21}\noindent{(\bf Bicomplex numbers)} 
		The set of bicomplex numbers $\C_2$ is defined as follows
		\begin{eqnarray*}
			\mathbb{C}_{2} &= & \{{u}_{1}  + {i}_{1} {u}_{2}  + {i}_{2} {u}_{3}  + {i}_{1} {i}_{2} {u}_{4}\;:\; {u}_{1} , {u}_{2} , {u}_{3} , {u}_{4} \in\mathbb{C}_{0}\},\end{eqnarray*}
		where $i_1$ and $i_2$ are unit vectors such that $i_1^2,i_2^2=-1$. Alternatively the set of bicomplex numbers is defined as follows
		\begin{eqnarray*}
			\mathbb{C}_{2} &= & \{{z}_{1} + {i}_{2} {z}_{2}\;:\;  {z}_{1}, {z}_{2} \in \mathbb{C}_{1}\},
		\end{eqnarray*}
			where $z_1, z_2$ are $i_1$ complex numbers and $i_1^2=-1$.
		For convenience, we used the symbols $\C_1,\C_0$  for the set of all complex numbers and the set of all real numbers respectively. In $\C_2$, there exist precisely four idempotent elements  $0,1,e_1,e_2$ in which $e_1,e_2$ are two nontrivial idempotent elements given as
		\[ e_{1} = \frac{(1 + i_{1} i_{2})}{2}\quad  \mbox{and} \quad e_{2} = \frac{(1 - i_{1} i_{2})}{2}.\]
		The numbers $e_1$ and $e_2$ serve as zero divisors, and they are linearly independent in the space $\mathbb{C}_2(\mathbb{C}_1)$. The linear independence of these elements over $\mathbb{C}_1$ gives rise to a new representation of all bicomplex numbers, known as the idempotent representation. Furthermore, for any element $\xi = {z}_{1}+ {i}_{2} {z}_{2}$ in $\C_2$, idempotent representation is defined as
		\[
		\xi =  (z_{1} -  i_{1} z_{2}) e_{1} + (z_{1}  +  i_{1} z_{2}) e_{2},
		\]
		where complex numbers $(z_{1} -  i_{1} z_{2})$ and $(z_{1} +  i_{1} z_{2})$ are called idempotent components of $\xi$ and denoted by $1_\xi$ and $2_\xi$ respectively. Thus, $\xi = {z}_{1}+ {i}_{2} {z}_{2}$ can also be written as $\xi = 1_\xi e_{1} + 2_\xi e_{2}$, where $1_\xi  = z_{1} - i_{1} z_{2}$ and  $2_\xi = z_{1}  + i_{1} z_{2}$.
    	\end{definition}
    
	\begin{definition}\label{22}
	\noindent{(\bf Cartesian product)} \cite{anjali2023matrix}
	The $n$-times Cartesian product of $\C_2$ is represented by $\C_2^n$ and defined as 
	\[
	{\C}_{2}^{n} =  \{(\xi_{1}, \xi_{2},\ldots,\xi_{n}) ; \;\xi_{i} \in \C_2; \;  i = 1,2,\ldots,n\}.
	\]
	Furthermore, by utilizing Idempotent representation of an element in $\C_2$, each element of $\C_2^n$ can also be uniquely expressed as
	\[
	(\xi_{1}, \xi_{2},\ldots,\xi_{n}) = (^1{\xi}_1,^1{\xi}_2, \cdots, ^1{\xi}_n)e_{1}  + (^2{\xi}_1,^2{\xi}_2, \cdots, ^2{\xi}_n)e_{2}
	\]
	such that $(^1{\xi}_1,^1{\xi}_2, \cdots, ^1{\xi}_n)$ and $(^2{\xi}_1,^2{\xi}_2, \cdots, ^2{\xi}_n)$ are  $n$-tuples of complex numbers from the space $\C_1^n$. For $(\xi_{1}, \xi_{2},\ldots,\xi_{n})$, $(\eta_{1}, \eta_{2},\ldots,\eta_{n})\in \C_2^n$ and $\kappa \in \C_2$, we have the following
	\begin{itemize}
		\item 
		$(\xi_{1}, \xi_{2},\ldots,\xi_{n})=(\eta_{1}, \eta_{2},\ldots, \eta_{n}) \Longleftrightarrow  {^1{\xi}_i} = {^1{\eta}_i}\;\& \;\; {^2{\xi}_i} = {^2{\eta}_i}\;\; \forall \; i$.
		\item 
		$(\xi_{1},\xi_{2},\ldots,\xi_{n})\cdot (\eta_{1}, \eta_{2},\ldots, \eta_{n})= (\xi_{1} \eta_{1}, \xi_{2} \eta_{2},\ldots, \xi_{n} \eta_{n})$.
		\item 
		$\kappa \cdot (\xi_{1},\xi_{2},\ldots,\xi_{n})= (\xi_{1},\xi_{2},\ldots,\xi_{n})  \cdot \kappa= (\kappa \xi_{1},\kappa \xi_{2},\ldots,\kappa \xi_{n})$.
	\end{itemize}
		Consequently, the product
	\begin{eqnarray*}
		(\xi_1,\xi_2,\ldots,\xi_n)e_i=(\xi_1e_i,\xi_2e_i,\ldots,\xi_n e_i)=(^i{\xi}_1e_i,^i{\xi}_2e_i,\ldots,^i{\xi}_n e_i) ; i=1,2.
	\end{eqnarray*}
	This conclusion arises from the fact that $\C_2^n$ is not only a $\C_2$ module but also exhibits the structure of a $\C_1$-algebra.
	\end{definition}

	\begin{definition}\noindent({\bf Bicomplex matrix}\label{23}) \cite{Amita2018}
		The set of bicomplex matrices of order $m\times n$ is denoted by $\C_2^{m\times n}$ and defined as
		\begin{eqnarray*}
			\C_2^{m\times n}=\left\{[\xi_{ij}]_{m\times n}:\xi_{ij}\in \C_2;1\leq i \leq m, 1\leq j \leq n \right\}.
		\end{eqnarray*}
		 Similar to the Idempotent representation of a bicomplex number, every bicomplex matrix can be uniquely decomposed as
		\[ 
		[\xi_{ij}]_{m\times n}=[^1{\xi}_{ij}]_{m\times n}e_1 + [^2{\xi}_{ij}]_{m\times n}e_2; 1\leq i \leq m, 1\leq j \leq n,  
		\]
		$i.e$. if $A=[\xi_{ij}]_{m\times n}$, then $A={^1{A}}e_1+{^2{A}}e_2$, where ${^1{A}}=[^1{\xi}_{ij}]_{m\times n}$ and ${^2{A}}=[^2{\xi}_{ij}]_{m\times n}$ are the complex matrices of order $m\times n$. The set $\C_2^{n\times n}$, under standard addition, multiplication, and scalar multiplication, forms an algebra over the field $\C_1$. 
	\end{definition}
\begin{definition}{\bf(Row rank of a matrix)}  \cite{wagh2024rank} 
If $A=[\xi_{ij}]_{m\times n} \in \C_2^{m \times n}$, then the row space of $A$ is defined as the subspace of vector-space $\C_2^n(\C_1)$ spanned by the rows of $A$ and the row rank of $A$ is denoted by $\rho_r(A)$ and defined as the dimension of row space of $A$.  
\end{definition} 
\begin{definition}{\bf(Column rank of a matrix)} \cite{wagh2024rank} 
If $A=[\xi_{ij}]_{m\times n} \in \C_2^{m \times n}$, then the column space of $A$ is defined as the subspace of vector-space $\C_2^m(\C_1)$ spanned by the columns of $A$ and the column rank of $A$ is denoted by $\rho_c(A)$ and defined as the dimension of column space of $A$.
\end{definition}
\begin{definition}({\bf Idempotent row rank of a matrix}) \cite{wagh2024rank} 
If $A = [\xi_{ij}]_{m\times n}\in \C_2^{m\times n}$, then the Idempotent row space of $A$ is defined as
\begin{eqnarray*}
\mbox{Idempotent row space of}\ A &=& (\mbox{row space of}\  {^1{A}}) e_1  + (\mbox{row space of}\  {^2{A}}) e_2\\
&=& \left\{  \left(\sum_{k=1}^{m}\alpha_k(^1{\xi}_{k1},^1{\xi}_{k2},\ldots,^1{\xi}_{kn})\right)e_1 \right. \\
&&\left.   +\left(\sum_{k=1}^{m}\beta_k(^2{\xi}_{k1},^2{\xi}_{k2},\ldots,^2{\xi}_{kn})\right)e_2;\ \alpha_k,\beta_k\in \C_1 \right\}.
\end{eqnarray*}
The Idempotent row rank of $A$ is denoted by $\rho_{ir}(A)$ and defined as the dimension of Idempotent row space of $A$.
\end{definition}

\begin{definition}({\bf Idempotent column rank of a matrix})\label{27} \cite{wagh2024rank} 
If $A = [\xi_{ij}]_{m\times n}\in \C_2^{m\times n}$, then the Idempotent column space of $A$ is defined as  
\begin{eqnarray*}
		\mbox{Idempotent column space of}\ A&=&(\mbox{column space of}\ {^1{A}}) e_1  + (\mbox{column space of}\  {^2{A}}) e_2\\
		&=& \left\{  \left(\sum_{k=1}^{n}\alpha_k(^1{\xi}_{1k},^1{\xi}_{2k}, \cdots, ^1{\xi}_{mk})\right)e_1 \right. \\
		&&\left.   +\left(\sum_{k=1}^{n}\beta_k(^2{\xi}_{1k},^2{\xi}_{2k}, \cdots, ^2{\xi}_{mk})\right)e_2;\ \alpha_k,\beta_k\in \C_1 \right\}.
	\end{eqnarray*}
	The Idempotent column rank of $A$ is denoted by $\rho_{ic}(A)$ and defined as the dimension of Idempotent column space of $A$.
\end{definition}
\begin{theorem} \cite{wagh2024rank} 
If $A=[\xi_{ij}]_{m\times n}\in \C_2^{m\times n}$, then row space of \ $A \subseteq$ Idempotent row space of $A$.
\end{theorem}
\begin{theorem}\label{2120}\cite{wagh2024rank} 
If $A=[\xi_{ij}]_{m\times n}\in \C_2^{m\times n}$, then 
column  space  of  $A  \subseteq$ Idempotent column space $A$.
\end{theorem}
\begin{theorem}\label{290} \cite{khanna2013}
Let ${T}: F^n\to$ $F^m$ be a linear transformation, where $F^n$ and $F^m$ be the vector spaces over the field $F$, and let ${\mathcal{B}}_{1}$ and ${\mathcal{B}}_{2}$ be the standard ordered bases of $F^{n}$ and $F^{m}$, respectively. Let the matrix representation of $T$ with respect to the ordered bases ${\mathcal{B}}_{1}$ and ${\mathcal{B}}_{2}$ be $A=[\xi_{ij}]_{m\times n} \in F^{m \times n}$, i.e. $A= [T]_{{{\mathcal{B}}_{1}}, {{\mathcal{B}}_{2}}}$ and $W$ be the subspace of $F^m$ such that  $W = \mbox{Column space of}\ A$, then $Range\left(T \right)=W$.
\end{theorem}
\begin{definition}\label{28} \cite{anjali2023matrix}
Let $(\xi_1,\xi_2, \cdots, \xi_n)$ = $(^1{\xi}_1,^1{\xi}_2, \cdots, ^1{\xi}_n)e_1$ + $(^2{\xi}_1,^2{\xi}_2, \cdots, ^2{\xi}_n)e_2 \in \C_2^n$ and $T_1$, $T_2\in L_1^{nm}$. Then a map $T:\C_2^n\rightarrow \C_2^m$ is defined by the rule
\begin{eqnarray*}
T(\xi_{1}, \xi_{2},\ldots,\xi_{n})= T_1(^1{\xi}_1,^1{\xi}_2, \cdots, ^1{\xi}_n)e_2+T_2(^2{\xi}_1,^2{\xi}_2, \cdots, ^2{\xi}_n)e_2,
\end{eqnarray*}
where $L_1^{nm}$ denotes the set of all $\C_1$ linear maps from $\C_1^n$ to $\C_1^m$.
It is a routine matter to check that $T$ is a linear transformation. This linear transformation $T$  is denoted by ${T}_{{1}}e_1+{T}_{{2}}e_2$, $i.e.$ ${T}={T}_{{1}}e_1+{T}_{{2}}e_2$. The set of all such linear maps is denoted by the idempotent product $L_1^{nm} \times_{e} L_1^{nm}$ and defined as
\begin{eqnarray*}
L_1^{nm} \times_{e} L_1^{nm}=\{T_1e_1+T_2e_2\in L_2^{nm}: T_1,T_2\in L_1^{nm}\},
\end{eqnarray*}
The idempotent product $L_1^{nm} \times_{e} L_1^{nm}$ is a subspace of $L_2^{nm}$ over the field $\C_1$, where $L_2^{nm}$ denotes the set of all $\C_1$ linear maps from $\C_2^n$ to $\C_2^m$. Also, (see \cite{anjali2023matrix}, 3.5) $\dim(L_1^{nm} \times_{e} L_1^{nm})(\C_1)=2mn$. The vector spaces $\C_2^{m\times n}$ and $L_1^{nm} \times_{e} L_1^{nm}$ are isomorphic, because they have same dimension over the field $\C_1$.   
\end{definition}
\begin{definition}
	{\bf(Matrix representation of $T=T_1e_1+T_2e_2$)}\label{29}
Let $T$ be a linear transformation such that  $T=T_1e_1+T_2e_2 \in L_1^{nm} \times_{e} L_1^{nm}$. Let  ${\mathcal{B}}_{1}$ and ${\mathcal{B}}_{2}$ be the ordered bases of $\C_2^{n}$ and $\C_2^{m}$, respectively. Then the matrix representation of $T$ with respect to the ordered bases ${\mathcal{B}}_{1}$ and ${\mathcal{B}}_{2}$ is a complex matrix $A=[\xi_{ij}]_{2m\times 2n}\in \C_1^{2m\times 2n}$. It is denoted by $[T]_{{{\mathcal{B}}_{1}}, {{\mathcal{B}}_{2}}}$, $i.e.$
\begin{eqnarray*}
	[T]_{{{\mathcal{B}}_{1}}, {{\mathcal{B}}_{2}}}= A=[\xi_{ij}]_{2m\times 2n}\in \C_1^{2m\times 2n}.
\end{eqnarray*}  
\end{definition} 

\begin{definition}{\bf(Idempotent matrix representation of $T=T_1e_1+T_2e_2$)}\label{210}\ \cite{anjali2023matrix} Let $T$ be a linear transformation such that  $T=T_1e_1+T_2e_2 \in L_1^{nm} \times_{e} L_1^{nm}$. Let  ${\mathcal{B}}_{1}$ and ${\mathcal{B}}_{2}$ be the ordered bases of $\C_1^{n}$ and $\C_1^{m}$, respectively. Then the Idempotent matrix representation of $T$ with respect to the ordered bases ${\mathcal{B}}_{1}$ and ${\mathcal{B}}_{2}$ is denoted by $[T]^{{\mathcal{B}}_{1}}_{{\mathcal{B}}_{2}}$ and defined as
\begin{eqnarray*}
[T]^{\mathcal{B}_{1}}_{\mathcal{B}_{2}}=  [T_1]^{\mathcal{B}_{1}}_{\mathcal{B}_{2}}e_1 +  [T_2]^{\mathcal{B}_{1}}_{\mathcal{B}_{2}}e_{2},
\end{eqnarray*}  
where $[T_1]^{\mathcal{B}_{1}}_{\mathcal{B}_{2}}$ and $[T_2]^{\mathcal{B}_{1}}_{\mathcal{B}_{2}}$ denote the matrix representation of $T_1$ and $T_2$ with respect to the ordered bases ${\mathcal{B}}_{1}$ and ${\mathcal{B}}_{2}$, respectively.
\end{definition}

\begin{theorem}\label{214} \cite{strang2022introduction}
If the entries of the matrix $A$ are from a field $F$, i.e. $A= [\xi_{ij}]_{m\times n}  \in {F}^{m\times n}$, $F^n$ and $F^m$ are the vector spaces of the field $F$, and ${T}_{A}: F^n\to$ $F^m$ is a map such that ${T}_{A}(X) = {A}_{m\times n} \ {X}_{n\times 1},\ \mbox{for}\ X = (x_{1}, x_{2}, \cdots, x_{n}) \in {F}^{n}$, then ${T}_{A}$ is a linear transformation from $F^n$ to $F^m$. Moreover, if ${A}_{m\times n} \ {X}_{n\times 1} = {0}_{m \times 1}$ is a system of m equations in n variables $x_{1}, x_{2}, \cdots, x_{n} \in F$, $i.e.$ $X = (x_{1}, x_{2}, \cdots, x_{n}) \in F^{n}$ and ${0}_{m \times 1} \in F^{m\times 1}$ is the null matrix, then $Y \in F^{n}$ is a solution of the system of equations  ${A}_{m\times n} \ {X}_{n\times 1} = {0}_{m\times 1}$ if and only if $Y \in \ker ({T}_{A})$.
\end{theorem}

\begin{theorem}\label{215} \cite{anjali2023matrix}
If $T= T_{1} e_{1}+ T_{2} e_{2}\in L_{1}^{nm} \times_{e} L_{1}^{nm}$, then
\begin{enumerate}
\item $\ker ( T_{1}e_{1} + T_{2}e_{2}) = \ker (T_{1}) \times_{e} \ker (T_{2}) $.
\item $range ( T_{1}e_{1} + T_{2}e_{2}) = range (T_{1}) \times_{e} range (T_{2})  $.
\end{enumerate}
\end{theorem}

\section{rank and solution of system of equations}
Anjali et al. \cite{anjali2023matrix}, had described an "Idempotent Method" for representing a matrix of a linear map of type $T=T_1e_1+T_2e_2: \C_2^{n} \longrightarrow \C_2^{m}$. This technique establishes a one-to-one correspondence between the bicomplex matrix $A$ and linear transformation $T= T_1e_1+T_2e_2 $ acting on the finite-dimensional vector spaces $\C_2^{n}\left( \C_1\right) $ to $\C_2^{m}\left( \C_1\right) $. Consequently, this linear transformation $T$ is crucial for our analysis and plays a significant role from the perspective of bicomplex matrices. Amita et al. \cite{wagh2024rank}
 had earlier investigated some results on the rank of bicomplex matrix. As $\C_{2}$ is not a field, so in this section, we have investigated a new approach to finding the solution of the system of equations ${A}_{m\times n} \ {X}_{n\times 1} = {B}_{m\times 1}$, where $A_{m\times n} \in \C_2^{m \times n}$, $X_{n\times 1} \in \C_{2}^{n\times 1}$, and $B_{m\times 1} \in \C_2^{m\times 1}$, and found some other results related to the rank of a bicomplex matrix.
\begin{definition}{\bf(Modified rank)}\label{31}
If $A=[\xi_{ij}]_{m\times n} \in \C_2^{m \times n}$, then the modified rank of $A$ is denoted by $\rho_{mr}(A)$ and defined as
\begin{align}
\mbox{Modified\ rank of}\ A &= rank({^{1}A}) + rank({^{2}A}) \\
i.e. \quad \rho_{mr}(A) &= \rho({^{1}A}) + \rho({^{2}A}).
\end{align}
\end{definition}
\begin{remark}\label{32} 
The definition \ref{31} yields that if $A=[\xi_{ij}]_{m\times n} \in \C_2^{m \times n}$, then $\rho_{mr}(A) \leq 2m, 2n $. 
\end{remark}
\begin{theorem}\label{33}
If $A={^{1}A} e_1 + {^{2}A} e_2 \in \C_2^{m\times n}$, then
dim$\left\lbrace (\mbox{column space of} \ {^{1}A})e_1 + (\mbox{column space of} \ {^{2}A})e_2 \right\rbrace$ = dim(column space of ${^{1}A}$) +  dim(column space of ${^{2}A}$), i.e. $\rho_{ic}(A)$ = $\rho({^{1}A}) + \rho({^{2}A})$.
\end{theorem}
\begin{proof}
Let $A={^{1}A} e_1 + {^{2}A} e_2 \in \C_2^{m\times n}$, ${\mathcal{B}_1}=\left\lbrace ^1{C}_1, ^1{C}_2,\cdots,^1{C}_{n_1} \right\rbrace  $ and ${\mathcal{B}_2}=\left\lbrace ^2{C}_1, ^2{C}_2,\cdots,^2{C}_{n_2} \right\rbrace  $ be the bases of column space of ${^{1}A}$ and column space of ${^{2}A}$, respectively.\\ 
Consider a set $S=\left\lbrace ^1{C}_1e_1, ^1{C}_2e_1,\cdots,^1{C}_{n_1}e_1, ^2{C}_1e_2, ^2{C}_2e_2,\cdots,^2{C}_{n_2}e_2\right\rbrace $. It is a trivial matter to find that
\begin{equation}\label{5}
\langle S \rangle \subseteq(\mbox{column space of} \ {^{1}A})e_1 + (\mbox{column space of} \ {^{2}A})e_2.
\end{equation}
Let
$(\xi_1,\xi_2, \cdots, \xi_m)\in (\mbox{column space of} \ {^{1}A})e_1 + (\mbox{column space of} \ {^{2}A})e_2$ such that
\[
 (\xi_1,\xi_2, \cdots, \xi_m) = (^1{\xi}_1,^1{\xi}_2, \cdots, ^1{\xi}_m)
 e_1 + (^2{\xi}_1,^2{\xi}_2, \cdots, ^2{\xi}_m)e_2.
 \] This implies that
\begin{eqnarray*} 
&&(^1{\xi}_1,^1{\xi}_2, \cdots, ^1{\xi}_m) \in (\mbox{column space of} \ {^{1}A})\ \mbox{and}\ (^2{\xi}_1,^2{\xi}_2, \cdots, ^2{\xi}_m) \in (\mbox{column space of} \ {^{2}A})\\
&\Rightarrow& (^1{\xi}_1,^1{\xi}_2, \cdots, ^1{\xi}_m)=\sum_{i=1}^{n_1}\alpha_i .^1{C}_{i}\ \mbox{and}\ (^2{\xi}_1,^2{\xi}_2,\cdots,^2{\xi}_m)=\sum_{j=1}^{n_2}\beta_j .^2{C}_j; \ \mbox{for} \  \alpha_i \ \mbox{and}\ \beta_j\in \C_1,\\
&&\mbox{for all}\ i \in \left\lbrace 1,2,\cdots, n_1\right\rbrace  \ \mbox{and}\ j \in \left\lbrace 1,2,\cdots, n_2 \right\rbrace \\
&\Rightarrow& (\xi_1,\xi_2, \cdots, \xi_m)=\left(\sum_{i=1}^{n_1}\alpha_i .^1{C}_{i}\right)e_1+\left(\sum_{j=1}^{n_2}\beta_j .^2{C}_j\right)e_2.\\
&& \mbox{In view of \ref{22}, we have}\\
&& (\xi_1,\xi_2, \cdots, \xi_m)= \sum_{i=1}^{n_1}\left(\alpha_i.^1{C}_{i}\right) e_1+\sum_{j=1}^{n_2}\left(\beta_j.^2{C}_{j}\right)e_2=\sum_{i=1}^{n_1}\alpha_i\left( ^1{C}_{i}e_1\right) +\sum_{j=1}^{n_2}\beta_j \left(^2{C}_je_2 \right)\\
&\Rightarrow& (\xi_1,\xi_2, \cdots, \xi_m)\in \langle S \rangle,
\end{eqnarray*}
which gives that 
\begin{equation}\label{6}
(\mbox{column space of} \ {^{1}A})e_1 + (\mbox{column space of} \ {^{2}A})e_2 \subseteq \langle S \rangle.
\end{equation}
Thus, from (\ref{5}) and (\ref{6}), we get
\begin{equation*}
(\mbox{column space of} \ {^{1}A})e_1 + (\mbox{column space of} \ {^{2}A})e_2 = \langle S \rangle.
\end{equation*}
Furthermore, if we consider 
\begin{eqnarray*}
&&\sum_{i=1}^{n_1}\alpha_i\left( ^1{C}_{i}e_1\right) +\sum_{j=1}^{n_2}\beta_j \left(^2{C}_je_2 \right)=0\\
&\Rightarrow& \sum_{i=1}^{n_1}\left(\alpha_i.^1{C}_{i}\right) e_1+\sum_{j=1}^{n_2}\left(\beta_j.^2{C}_{j}\right)e_2  =0\\
&\Rightarrow& \left(\sum_{i=1}^{n_1}\alpha_i .^1{C}_{i}\right)e_1+\left(\sum_{j=1}^{n_2}\beta_j .^2{C}_j\right)e_2=0,
\end{eqnarray*}
using \ref{22} and, ${\mathcal{B}_1}$, ${\mathcal{B}_2}$ being the bases follows that
\begin{eqnarray*}
&& \left(\sum_{i=1}^{n_1}\alpha_i .^1{C}_{i}\right)=0\ \mbox{and}\ \left(\sum_{j=1}^{n_2}\beta_j .^2{C}_j\right)=0\\
&\Rightarrow& \alpha_i=0 \ \mbox{and} \ \beta_j=0, \ \mbox{for all}\ i \in \left\lbrace 1,2,\cdots, n_1\right\rbrace  \ \mbox{and}\ j \in \left\lbrace 1,2,\cdots, n_2 \right\rbrace. 
\end{eqnarray*}
This implies that $S$ is linearly independent set. Therefore, $S$ is a basis of (column space of ${^{1}A})e_1$ + (column space of ${^{2}A})e_2$. Hence, dim$\left\lbrace (\mbox{column space of} \ {^{1}A})e_1 + (\mbox{column space of} \ {^{2}A})e_2 \right\rbrace$ = dim(column space of ${^{1}A}$) +  dim(column space of ${^{2}A}$).
\end{proof}
Dually we may prove the following corollary for idempotent row space of $A$.
\begin{corollary}\label{34}
	If $A={^{1}A} e_1 + {^{2}A} e_2 \in \C_2^{m\times n}$, then
dim$\left\lbrace (\mbox{row space of} \ {^{1}A})e_1 + (\mbox{row space of} \ {^{2}A})e_2 \right\rbrace$ = dim(row space of ${^{1}A}$) +  dim(row space of ${^{2}A}$), i.e. $\rho_{ir}(A)$ = $\rho({^{1}A}) + \rho({^{2}A})$.
\end{corollary}
The following corollaries \ref{35} and \ref{36} are immediate consequences of \ref{31}, \ref{33} and \ref{34}.
\begin{corollary}\label{35}
If $A={^{1}A} e_1 + {^{2}A} e_2 \in \C_2^{m\times n}$, then
\begin{eqnarray*}
\mbox{dim}\left\lbrace (\mbox{row space of} \ {^{1}A})e_1 + (\mbox{row space of} \ {^{2}A})e_2 \right\rbrace&=&\mbox{dim}\left\lbrace (\mbox{column space of} \ {^{1}A})e_1 + (\mbox{column space of} \ {^{2}A})e_2 \right\rbrace\\
i.e. \quad \rho_{ir}(A)&=&\rho_{ic}(A)\ =\ \rho({^{1}A}) + \rho({^{2}A}).
\end{eqnarray*}
\end{corollary}
\begin{corollary}\label{36}
If $A={^{1}A} e_1 + {^{2}A} e_2 \in \C_2^{m\times n}$, then
\begin{eqnarray*}
\mbox{Idempotent row rank of}\ A&=&\mbox{Idempotent column rank of}\ A \ =\ \mbox{Modified\ rank of}\ A.\\
i.e. \quad \rho_{ir}(A)&=&\rho_{ic}(A)\ =\ \rho_{mr}(A)\ = \ \rho({^{1}A}) + \rho({^{2}A}).
\end{eqnarray*}
\end{corollary}
If ${T}:\C_2^n\to$ $\C_2^m$ is a linear transformation such that  $T=T_1e_1+T_2e_2 \in L_1^{nm} \times_{e} L_1^{nm}$, then we cannot apply \ref{290} on $T$ as $\C_{2}$ is not a field. Thus, with the help of Idempotent matrix representation of $T$, we have established a new result \ref{37} for such $T$.  
\begin{theorem}\label{37}
Let $T$ be a linear transformation such that  $T=T_1e_1+T_2e_2 \in L_1^{nm} \times_{e} L_1^{nm}$, and ${\mathcal{B}}_{1}$ and ${\mathcal{B}}_{2}$ be the standard ordered bases of $\C_1^{n}$ and $\C_1^{m}$, respectively. Let the Idempotent matrix representation of $T$ with respect to the ordered bases ${\mathcal{B}}_{1}$ and ${\mathcal{B}}_{2}$ be $A=[\xi_{ij}]_{m\times n} \in \C_2^{m \times n}$, $i.e. \ A= [T]^{{\mathcal{B}}_{1}}_{{\mathcal{B}}_{2}}$ and $W$ be the subspace of $\C_2^m$ such that  $W = \mbox{Idempotent column space of}\ A$. Then column  space  of  $A  \subseteq Range\left(T \right)=W$.
\end{theorem}
\begin{proof}
Let ${\mathcal{B}}_{1}=\left\lbrace f_1, f_2, f_3, \cdots, f_n \right\rbrace $ and ${\mathcal{B}}_{2}=\left\lbrace g_1, g_2, g_3, \cdots, g_m \right\rbrace$ be the standard ordered bases of $\C_1^{n}$ and $\C_1^{m}$, respectively and  $T=T_1e_1+T_2e_2 \in L_1^{nm} \times_{e} L_1^{nm}$ such that $[T]^{{\mathcal{B}}_{1}}_{{\mathcal{B}}_{2}}=A=[\xi_{ij}]_{m\times n}$. Using \ref{23}, \ref{2120}, \ref{28} and \ref{210}, we get
\begin{eqnarray*}
&& [T_1e_1+T_2e_2]^{{\mathcal{B}}_{1}}_{{\mathcal{B}}_{2}} = A=[\xi_{ij}]_{m\times n}\\
&\Rightarrow& [T_1]^{{\mathcal{B}}_{1}}_{{\mathcal{B}}_{2}} e_1 +  [T_2]^{{\mathcal{B}}_{1}}_{{\mathcal{B}}_{2}} e_{2}= A=[\xi_{ij}]_{m\times n}\\
&\Rightarrow& [T_1]^{{\mathcal{B}}_{1}}_{{\mathcal{B}}_{2}} = {^{1}A}=[^1{\xi}_{ij}]_{m\times n}\ \mbox{and} \  [T_2]^{{\mathcal{B}}_{1}}_{{\mathcal{B}}_{2}} = {^{2}A}=[^2{\xi}_{ij}]_{m\times n}\\
&\Rightarrow&  {T_1}\left( f_j \right) =\sum_{i=1}^{m}{^1{\xi}_{ij}}g_i \ \mbox{and} \ {T_2}\left( f_j \right) =\sum_{i=1}^{m}{^2{\xi}_{ij}}g_i;\ \mbox{where} \ j=1, 2, \cdots, n.
\end{eqnarray*}
Now, we assert that $T:\C_2^n\to$ $W$ is an onto linear transformation. Let us suppose that $X={^{1}X}e_1+{^{2}X}e_2\in \C_2^n$. Then
\begin{eqnarray*}
{^{1}X}=\sum_{i=1}^{n} \alpha_{i} f_{i} \quad \mbox{and} \quad {^{2}X}=\sum_{i=1}^{n} \beta_{i} f_{i};\ \mbox{for all} \  \alpha_i \ \text{and}\ \beta_i\in \C_1,\ i\in \left\lbrace 1,2,\cdots, n\right\rbrace.
\end{eqnarray*} 
Now,
\begin{eqnarray*}
T(X)&=&(T_{1}e_{1} + T_{2} e_{2})(X)\\
&=& T_{1}({^{1}X})e_{1} + T_{2}({^{2}X})e_{2}\\
&=&  T_{1}\left(\sum_{i=1}^{n} \alpha_{i} f_{i}\right)e_{1} +  T_{2}\left(\sum_{i=1}^{n} \beta_{i} f_{i}\right) e_{2}\\
&=&   \left(\sum_{i=1}^{n} \alpha_{i} T_{1}(f_{i})\right) e_{1}+   \left(\sum_{i=1}^{n} \beta_{i} T_{2}(f_{i})\right)e_{2}
\end{eqnarray*}
It is evident that $T_{1}(f_{j})=\left({^1{\xi}_{1j}}, {^1{\xi}_{2j}}, \cdots, {^1{\xi}_{mj}} \right)\  \mbox{and}\ T_{2}(f_{j})=\left({^2{\xi}_{1j}}, {^2{\xi}_{2j}}, \cdots, {^2{\xi}_{mj}} \right)$, $\forall \ j\in \left\lbrace 1,2,\cdots, n\right\rbrace$. Then
\begin{eqnarray*}
&& \left(\sum_{i=1}^{n} \alpha_{i} T_{1}(f_{i})\right) \in \mbox{column space of} \ {^{1}A} \quad \mbox{and} \quad \left(\sum_{i=1}^{n} \beta_{i} T_{2}(f_{i})\right) \in \mbox{column space of} \ {^{2}A}\\
&\Rightarrow& T(X)\in\mbox{Idempotent column space of}\ A=W.
\end{eqnarray*}
Furthermore, if we take $Y={^{1}Y}e_1+{^{2}Y}e_2\in W$. Then
\begin{eqnarray*}
&& {^{1}Y}\in \mbox{column space of} \ {^{1}A} \quad \mbox{and} \quad {^{2}Y} \in \mbox{column space of} \ {^{2}A}  \\
&\Rightarrow& {^{1}Y} = \left(\sum_{i=1}^{n} \gamma_{i} T_{1}(f_{i})\right)
\quad \mbox{and} \quad {^{2}Y} = \left(\sum_{i=1}^{n} \delta_{i} T_{2}(f_{i})\right);\ \mbox{for all} \  \gamma_i \ \text{and}\ \delta_i\in \C_1,\ i\in \left\lbrace 1,2,\cdots, n\right\rbrace\\
&\Rightarrow& {^{1}Y}e_1+{^{2}Y}e_2 = \left(\sum_{i=1}^{n} \gamma_{i} T_{1}(f_{i})\right) e_1 + \left(\sum_{i=1}^{n} \delta_{i} T_{2}(f_{i})\right)e_2\\
&\Rightarrow& {^{1}Y}e_1+{^{2}Y}e_2 = T_{1}\left(\sum_{i=1}^{n} \gamma_{i} f_{i}\right)e_1  +  T_{2}\left(\sum_{i=1}^{n} \delta_{i} f_{i}\right)e_2\\
&\Rightarrow& {^{1}Y}e_1+{^{2}Y}e_2 = ( T_{1} e_1+ T_{2}e_2) \left[\left(\sum_{i=1}^{n} \gamma_{i} f_{i}\right)e_1  + \left(\sum_{i=1}^{n}  \delta_{i} f_{i}\right)e_2 \right]\\
&\Rightarrow& Y = T \left[\left(\sum_{i=1}^{n} \gamma_{i} f_{i}\right)e_1  + \left(\sum_{i=1}^{n}  \delta_{i} f_{i}\right)e_2 \right]
\end{eqnarray*}
It shows that $T:\C_2^n\to$ $W$ is an onto linear transformation. Therefore
\begin{center}
column  space  of  $A  \subseteq Range\left(T \right)=W$.   
\end{center}
\end{proof}
The following corollary \ref{38} is immediate consequence of \ref{36} and \ref{37}.
\begin{corollary}\label{38}
Let $T$ be a linear transformation such that  $T=T_1e_1+T_2e_2 \in L_1^{nm} \times_{e} L_1^{nm}$, and let ${\mathcal{B}}_{1}$ and ${\mathcal{B}}_{2}$ be the standard ordered bases of $\C_1^{n}$ and $\C_1^{m}$, respectively. Let the Idempotent matrix representation of $T$ with respect to the ordered bases ${\mathcal{B}}_{1}$ and ${\mathcal{B}}_{2}$ be $A=[\xi_{ij}]_{m\times n} \in \C_2^{m \times n}$, $i.e.$ $A= [T]^{{\mathcal{B}}_{1}}_{{\mathcal{B}}_{2}}$. Then
\begin{enumerate}
\item
$Rank(T)=$ dim(Idempotent column space of $A$) i.e. $\rho(T)= \rho_{ic}(A)$.
\item 
$\rho_{ir}(A)=\rho_{ic}(A)=\rho_{mr}(A)=\rho({^{1}A}) + \rho({^{2}A}) = \rho(T)$.
\end{enumerate}
\end{corollary}
\begin{definition}
Let $A={^{1}A} e_1 + {^{2}A} e_2 = [\xi_{ij}]_{m\times n} \in \C_2^{m\times n}$, and ${^{1}X}\ \mbox{and}\ {^{2}X}\in \C_1^n$. Then we define ${T}_{^{1}A}:\C_1^n\to$ $\C_1^m$ and ${T}_{^{2}A}:\C_1^n\to$ $\C_1^m$ such that ${T}_{^{1}A}({^{1}X}) = {^{1}A}_{m\times n} {^{1}X}_{n\times 1}$ and ${T}_{^{2}A}({^{2}X}) = {^{2}A}_{m\times n} {^{2}X}_{n\times 1}$. It is a routine matter to check that ${T}_{^{1}A}e_1+{T}_{^{2}A}e_2 \in L_1^{nm} \times_{e} L_1^{nm}$. This linear transformation ${T}_{^{1}A}e_1+{T}_{^{2}A}e_2$ is denoted by ${T}_{A}$, $i.e.$ ${T}_{A}={T}_{^{1}A}e_1+{T}_{^{2}A}e_2$.
\end{definition}
\begin{remark}\label{310}
If $A={^{1}A} e_1 + {^{2}A} e_2 = [\xi_{ij}]_{m\times n} \in \C_2^{m\times n}$ and $X={^{1}X}e_1+{^{2}X}e_2\in \C_2^n$, then ${T}_{A}:\C_2^n\to$ $\C_2^m$ and
\begin{eqnarray*}
{T}_{A}(X) &=& {T}_{^{1}A}({^{1}X}){e_1} + {T}_{^{2}A}({^{2}X}){e_2}\\ 
&=& \left({^{1}A}_{m\times n} {^{1}X}_{n\times 1}\right) {e_1} + \left({^{2}A}_{m\times n} {^{2}X}_{n\times 1}\right) {e_2}\\
&=& {A}_{m\times n} \ {X}_{n\times 1}.
\end{eqnarray*}
Moreover, if ${\mathcal{B}}_{1}$ and ${\mathcal{B}}_{2}$ are the standard ordered bases of $\C_1^{n}$ and $\C_1^{m}$, respectively, then the idempotent matrix representation of ${T}_{A}$ with respect to the ordered bases ${\mathcal{B}}_{1}$ and ${\mathcal{B}}_{2}$ is the matrix $A$, i.e. $[{T}_{A}]^{{\mathcal{B}}_{1}}_{{\mathcal{B}}_{2}} = A$. By \ref{37}\ \mbox{and}\ \ref{38}, it is evident that $Range\left({T}_{A} \right) = \mbox{Idempotent column space of}\ A$ and $\rho({T}_{A}) = \rho({^{1}A}) + \rho({^{2}A}) = \rho_{ir}(A)=\rho_{ic}(A)=\rho_{mr}(A)$.
\end{remark}
\begin{theorem}\label{311}
Let $A= {^{1}A} e_1 + {^{2}A} e_2 =[\xi_{ij}]_{m\times n}  \in \C_2^{m\times n}, \ B = {^{1}B} e_1 + {^{2}B} e_2 = (b_{1}, b_{2}, \cdots, b_{m}) \in \C_{2}^{m}$  and ${A}_{m\times n} \ {X}_{n\times 1} = {B}_{m\times 1}$ be a system of m equations in n variables $x_{1}, x_{2}, \cdots, x_{n} \in \C_{2}$, that is $X = (x_{1}, x_{2}, \cdots, x_{n}) \in \C_{2}^{n}$ and
\begin{eqnarray*}
&\xi_{11}{x_{1}} + \xi_{12}{x_{2}} + \xi_{13}{x_{3}} +  \cdots + \xi_{1n}{x_{n}}& = b_{1},\\
&\xi_{21}{x_{1}} + \xi_{22}{x_{2}} + \xi_{23}{x_{3}} +  \cdots + \xi_{2n}{x_{n}}& = b_{2},\\
&\xi_{31}{x_{1}} + \xi_{32}{x_{2}} + \xi_{33}{x_{3}} +  \cdots + \xi_{3n}{x_{n}}& = b_{3},\\
&\vdots \hspace{0.45cm} \vdots \hspace{0.45cm} \vdots \hspace{0.45cm} \vdots \hspace{0.45cm} \vdots \hspace{0.45cm} \vdots \hspace{0.45cm} \vdots \hspace{0.45cm}\vdots \hspace{0.45cm}\vdots \hspace{0.45cm} \vdots& \vdots \hspace{0.45cm} \vdots\\
&\xi_{m1}{x_{1}} + \xi_{m2}{x_{2}} + \xi_{m3}{x_{3}} +  \cdots + \xi_{mn}{x_{n}}& = b_{m},
\end{eqnarray*}
then there exists a solution of the system of equations  ${A}_{m\times n} \ {X}_{n\times 1} = {B}_{m\times 1}$ if and only if $\rho({^{1}A}\ | \ {^{1}B})$ = $\rho({^{1}A})$ and $\rho({^{2}A}\ | \ {^{2}B})$ = $\rho({^{2}A})$.
\end{theorem}
\begin{proof}
Suppose that $Y = {^{1}Y}e_1+{^{2}Y}e_2 \in \C_{2}^{n}$ is a solution of the system of equations  ${A}_{m\times n} \ {X}_{n\times 1} = {B}_{m\times 1}$. Using \ref{27} and \ref{310}, we have
\begin{eqnarray*}
&& {A}_{m\times n} \ {Y}_{n\times 1} = {B}_{m\times 1}\\
&\Rightarrow& {T}_{A}(Y) = {B}\\
&\Rightarrow& {B} \in Range\left({T}_{A} \right)\\
&\Rightarrow& {B} \in \mbox{Idempotent column space of}\ A\\
&\Rightarrow& {B} \in (\mbox{column space of} \ {^{1}A})e_1 + (\mbox{column space of} \ {^{2}A})e_2\\
&\Rightarrow& {^{1}B} \in \mbox{column space of} \ {^{1}A}\ \ \mbox{and}\ \ {^{2}B} \in \mbox{column space of} \ {^{2}A}.
\end{eqnarray*}
Therefore, ${^{1}B}$ is linearly dependent on the column vectors $ \left({^1{\xi}_{11}}, {^1{\xi}_{21}}, \cdots, {^1{\xi}_{m1}} \right)$, $\left({^1{\xi}_{12}}, {^1{\xi}_{22}}, \cdots, {^1{\xi}_{m2}} \right)$, $\cdots$, $\left({^1{\xi}_{1n}}, {^1{\xi}_{2n}}, \cdots, {^1{\xi}_{mn}} \right)$ of  ${^{1}A}$ and  ${^{2}B}$ is linearly dependent on the column vectors 
 $\left({^2{\xi}_{11}}, {^2{\xi}_{21}}, \cdots, {^2{\xi}_{m1}} \right)$,\\ $\left({^2{\xi}_{12}}, {^2{\xi}_{22}}, \cdots, {^2{\xi}_{m2}} \right)$, $\cdots$, $\left({^2{\xi}_{1n}}, {^2{\xi}_{2n}}, \cdots, {^2{\xi}_{mn}} \right)$ of ${^{2}A}$. Hence, $\rho({^{1}A}\ | \ {^{1}B})$ = $\rho({^{1}A})$ and $\rho({^{2}A}\ | \ {^{2}B})$ = $\rho({^{2}A})$.\\
Converse: Let $\rho({^{1}A}\ | \ {^{1}B})$ = $\rho({^{1}A})$ and $\rho({^{2}A}\ | \ {^{2}B})$ = $\rho({^{2}A})$. Then ${^{1}B}$ is linearly dependent on the column vectors $ \left({^1{\xi}_{11}}, {^1{\xi}_{21}}, \cdots, {^1{\xi}_{m1}} \right)$, $\left({^1{\xi}_{12}}, {^1{\xi}_{22}}, \cdots, {^1{\xi}_{m2}} \right)$, $\cdots$, $\left({^1{\xi}_{1n}}, {^1{\xi}_{2n}}, \cdots, {^1{\xi}_{mn}} \right)$ of ${^{1}A}$ and ${^{2}B}$ is linearly dependent on the column vectors  $\left({^2{\xi}_{11}}, {^2{\xi}_{21}}, \cdots, {^2{\xi}_{m1}} \right)$, $\left({^2{\xi}_{12}}, {^2{\xi}_{22}}, \cdots, {^2{\xi}_{m2}} \right)$, $\cdots$, $\left({^2{\xi}_{1n}}, {^2{\xi}_{2n}}, \cdots, {^2{\xi}_{mn}} \right)$ of ${^{2}A}$. Therefore,
\begin{eqnarray*}
&& {^{1}B} \in \mbox{column space of} \ {^{1}A}\ \ \mbox{and}\ \ {^{2}B} \in \mbox{column space of} \ {^{2}A}\\
&\Rightarrow& {B} \in (\mbox{column space of} \ {^{1}A})e_1 + (\mbox{column space of} \ {^{2}A})e_2\\
&\Rightarrow& {B} \in \mbox{Idempotent column space of}\ A\\
&\Rightarrow& {B} \in Range\left({T}_{A} \right)\\
&\Rightarrow& \mbox{there exists}\ Y \in \C_{2}^{n}\ \mbox{such that}\ {T}_{A}(Y) = {B}\\
&\Rightarrow& {A}_{m\times n} \ {Y}_{n\times 1} = {B}_{m\times 1}.
\end{eqnarray*}
\end{proof}
\begin{theorem}\label{312}
Let $A= {^{1}A} e_1 + {^{2}A} e_2 =[\xi_{ij}]_{m\times n}  \in \C_2^{m\times n}$ and $\ B = {^{1}B} e_1 + {^{2}B} e_2 = (b_{1}, b_{2}, \cdots, b_{m}) \in \C_{2}^{m}$. Then $\rho({^{1}A}\ | \ {^{1}B})$ = $\rho({^{1}A})$ and $\rho({^{2}A}\ | \ {^{2}B})$ = $\rho({^{2}A})$ if and only if $\rho_{mr}({A}\ | \ {B})$ = $\rho_{mr}({A})$.
\end{theorem}
\begin{proof}
Let us suppose that $\rho({^{1}A}\ | \ {^{1}B})$ = $\rho({^{1}A})$ and $\rho({^{2}A}\ | \ {^{2}B})$ = $\rho({^{2}A})$, then using \ref{31}, we have
\begin{eqnarray*}
\rho_{mr}({A}\ | \ {B}) &=& \rho({^{1}A}\ | \ {^{1}B}) + \rho({^{2}A}\ | \ {^{2}B})\\
&=& \rho({^{1}A}) + \rho({^{2}A})\\
&=& \rho_{mr}({A}).
\end{eqnarray*}
Hence, $\rho_{mr}({A}\ | \ {B})$ = $\rho_{mr}({A})$.\\
Converse: Suppose that $\rho_{mr}({A}\ | \ {B})$ = $\rho_{mr}({A})$, $\rho({^{1}A}) = t_{1}$ and $\rho({^{2}A}) = t_{2}$, for $t_{1}\ \mbox{and}\ t_{2} \in \N \cup \left\lbrace 0 \right\rbrace$, then we have 
\begin{eqnarray*}
\rho_{mr}({A}\ | \ {B}) &=& \rho({^{1}A}) + \rho({^{2}A})\\
\Rightarrow \rho({^{1}A}\ | \ {^{1}B}) + \rho({^{2}A}\ | \ {^{2}B}) &=& t_{1} + t_{2}.
\end{eqnarray*}
It is evident that $t_{1} \le \rho({^{1}A}\ | \ {^{1}B}) \le t_{1} + 1$ and $t_{2} \le \rho({^{2}A}\ | \ {^{2}B}) \le t_{2} + 1$. Therefore, $\rho({^{1}A}\ | \ {^{1}B}) = t_{1}$ and $\rho({^{2}A}\ | \ {^{2}B}) = t_{2}$. Hence, $\rho({^{1}A}\ | \ {^{1}B})$ = $\rho({^{1}A})$ and $\rho({^{2}A}\ | \ {^{2}B})$ = $\rho({^{2}A})$.
\end{proof}

The following corollary \ref{313} is immediate consequence of \ref{311} and \ref{312}.
\begin{corollary}\label{313}
Let $A= {^{1}A} e_1 + {^{2}A} e_2 =[\xi_{ij}]_{m\times n}  \in \C_2^{m\times n}, \ B = {^{1}B} e_1 + {^{2}B} e_2 = (b_{1}, b_{2}, \cdots, b_{m}) \in \C_{2}^{m}$  and ${A}_{m\times n} \ {X}_{n\times 1} = {B}_{m\times 1}$ be a system of m equations in n variables $x_{1}, x_{2}, \cdots, x_{n} \in \C_{2}$, i.e. $X = (x_{1}, x_{2}, \cdots, x_{n}) \in \C_{2}^{n}$. Then the following conditions are equivalent:
\begin{enumerate}
\item There exists a solution of the system of equations  ${A}_{m\times n} \ {X}_{n\times 1} = {B}_{m\times 1}$.
\item $\rho({^{1}A}\ | \ {^{1}B})$ = $\rho({^{1}A})$ and $\rho({^{2}A}\ | \ {^{2}B})$ = $\rho({^{2}A})$.
\item $\rho_{mr}({A}\ | \ {B})$ = $\rho_{mr}({A})$.
\end{enumerate}
\end{corollary}
\begin{remark}\label{314A}
As the entries of the matrix $A= {^{1}A} e_1 + {^{2}A} e_2 =[\xi_{ij}]_{m\times n}  \in \C_2^{m\times n}$ are not from a field, it is a routine matter to check whether the result given in \ref{214} is valid for $A= {^{1}A} e_1 + {^{2}A} e_2 =[\xi_{ij}]_{m\times n}  \in \C_2^{m\times n}$. Theorem \ref{215} and Remark \ref{310} yield that if $A= {^{1}A} e_1 + {^{2}A} e_2 =[\xi_{ij}]_{m\times n}  \in \C_2^{m\times n}$ and ${A}_{m\times n} \ {X}_{n\times 1} = {0}_{m\times 1}$ is a system of m equations in n variables $x_{1}, x_{2}, \cdots, x_{n} \in \C_{2}$, $i.e.$ $X = (x_{1}, x_{2}, \cdots, x_{n}) \in \C_{2}^{n}$ and ${0}_{m\times 1} \in \C_2^{m\times 1}$ is the null matrix. Then the following conditions are equivalent:
\begin{enumerate}
\item $Y = {^{1}Y}e_1+{^{2}Y}e_2 \in \C_{2}^{n}$ is a solution of the system of equations  ${A}_{m\times n} \ {X}_{n\times 1} = {0}_{m\times 1}$.
\item $Y \in \ker ({T}_{A})$.
\item ${^{1}Y} \in \ker ({T}_{^{1}A})$ and ${^{2}Y} \in \ker ({T}_{^{2}A})$.
\end{enumerate}
Moreover, the set of solution of the system of equations ${A}_{m\times n} \ {X}_{n\times 1} = {0}_{m\times 1}$ is $\ker ({T}_{A}) = \ker ({T}_{^{1}A}) \times_{e} \ker ({T}_{^{2}A}) = \ker ({T}_{^{1}A})e_1 + \ker ({T}_{^{2}A})e_2$. It is evident that the system of equations ${A}_{m\times n} \ {X}_{n\times 1} = {0}_{m\times 1}$ has a unique solution or an infinite number of solutions.
\end{remark}
\begin{theorem}\label{315A}
Let $A= {^{1}A} e_1 + {^{2}A} e_2 =[\xi_{ij}]_{m\times n}  \in \C_2^{m\times n}$ and ${A}_{m\times n} \ {X}_{n\times 1} = {0}_{m\times 1}$ be a system of $m$ equations in $n$ variables $x_{1}, x_{2}, \cdots, x_{n} \in \C_{2}$, i.e. $X = (x_{1}, x_{2}, \cdots, x_{n}) \in \C_{2}^{n}$. Then we have the following:
\begin{enumerate}
	\item If $\rho_{mr}(A) = 2n$, then the system of equations ${A}_{m\times n} \ {X}_{n\times 1} = {0}_{m\times 1}$ has a unique solution.
\item If $\rho_{mr}(A) = k, k \neq 2n$, then the system of equations ${A}_{m\times n} \ {X}_{n\times 1} = {0}_{m\times 1}$ has an infinite number of solutions.
\item If $\rho_{mr}(A) = 0$, then every element of $\C_{2}^{n}$ is a solution of the system of equations ${A}_{m\times n} \ {X}_{n\times 1} = {0}_{m\times 1}$.
\end{enumerate}
\end{theorem}
\begin{proof}
Suppose that $A= {^{1}A} e_1 + {^{2}A} e_2 =[\xi_{ij}]_{m\times n}  \in \C_2^{m\times n}$ and ${A}_{m\times n} \ {X}_{n\times 1} = {0}_{m\times 1}$ is a system of $m$ equations in $n$ variables $x_{1}, x_{2}, \cdots, x_{n} \in \C_{2}$, $i.e.$ $X = (x_{1}, x_{2}, \cdots, x_{n}) \in \C_{2}^{n}$. By using \ref{310} and \ref{314A}, we have 
\begin{eqnarray*}
&& {T}_{A}:\C_2^n\to \C_2^m \ \mbox{such that}\ {T}_{A}(X) ={A}_{m\times n} \ {X}_{n\times 1}\ \mbox{and}\ 2n = \rho({T}_{A}) + \eta({T}_{A}),
\end{eqnarray*}
where $\eta({T}_{A})$ represents $dim(\ker ({T}_{A}))$. Therefore, $2n = \rho_{mr}(A) + \eta({T}_{A})$.
\begin{enumerate}
\item 
If $\rho_{mr}(A) = 2n$, then $\eta({T}_{A}) = 0$. Thus,
\begin{eqnarray*}
 \ker ({T}_{A}) = \left\lbrace (0, 0, \cdots, 0) \right\rbrace \subset\C_{2}^{n}.
\end{eqnarray*}
Therefore, the solution of the system of equations ${A}_{m\times n} \ {X}_{n\times 1} = {0}_{m\times 1}$ is only $(0, 0, \cdots, 0) \in\C_{2}^{n}$.
\item 
If $\rho_{mr}(A) = k, k \neq 2n$ then $\eta({T}_{A}) = 2n - k > 0$. Thus,
\begin{eqnarray*}
 \ker ({T}_{A}) \neq \left\lbrace (0, 0, \cdots, 0) \right\rbrace \ \mbox{and has an infinite number of elements}.
\end{eqnarray*}
Therefore, the system of equations ${A}_{m\times n} \ {X}_{n\times 1} = {0}_{m\times 1}$ has an infinite number of solutions.
\item 
If $\rho_{mr}(A) = 0$, then $\eta({T}_{A}) = 2n$. Thus,
\begin{eqnarray*}
 \ker ({T}_{A}) = \C_{2}^{n}.
\end{eqnarray*}
Therefore, every element of $\C_{2}^{n}$ is a solution of the system of equations ${A}_{m\times n} \ {X}_{n\times 1} = {0}_{m\times 1}$.
\end{enumerate}
\end{proof}
\begin{example}
If $A=[\xi_{ij}]_{3\times 5}  \in \C_2^{3\times 5}$ and ${A}_{3\times 5} \ {X}_{5\times 1} = {0}_{3\times 1}$ is a system of three equations in five variables $x_{1}, x_{2}, x_{3}, x_{4}, x_{5} \in \C_{2}$, $i.e.$ $X = (x_{1}, x_{2}, x_{3}, x_{4}, x_{5}) \in \C_{2}^{5}$. By using \ref{32}, $\rho_{mr}(A) \le 6$ $ $i.e.$ $ $\rho_{mr}(A) \neq 10$. Therefore, the system of equations ${A}_{3\times 5} \ {X}_{5\times 1} = {0}_{3\times 1}$ has an infinite number of solutions.
\end{example}
\begin{theorem}\label{315}
Let $A= {^{1}A} e_1 + {^{2}A} e_2 =[\xi_{ij}]_{m\times n}  \in \C_2^{m\times n}, \ B = {^{1}B} e_1 + {^{2}B} e_2 = (b_{1}, b_{2}, \cdots, b_{m}) \in \C_{2}^{m}$  and $S = {^{1}S}e_1+{^{2}S}e_2 \in \C_{2}^{n}$ be a solution of the system ${A}_{m\times n} \ {X}_{n\times 1} = {B}_{m\times 1}$ of $m$ equations in $n$ variables $x_{1}, x_{2}, \cdots, x_{n} \in \C_{2}$, $i.e.$ $X = (x_{1}, x_{2}, \cdots, x_{n}) \in \C_{2}^{n}$. Then $Y = {^{1}Y}e_1+{^{2}Y}e_2 \in \C_{2}^{n}$ is a solution of the system of equations  ${A}_{m\times n} \ {X}_{n\times 1} = {B}_{m\times 1}$ if and only if $Y \in \ker ({T}_{A}) + S$.
\end{theorem}
\begin{proof}
Suppose that $Y = {^{1}Y}e_1+{^{2}Y}e_2 \in \C_{2}^{n}$ is a solution of the system of equations  ${A}_{m\times n} \ {X}_{n\times 1} = {B}_{m\times 1}$ and
\begin{eqnarray*}
\ker ({T}_{A}) + S = \left\lbrace U + S : U \in \ker ({T}_{A})\right\rbrace.
\end{eqnarray*}
Therefore, we have
\begin{eqnarray*}
&& {A}_{m\times n} \ {Y}_{n\times 1} = {B}_{m\times 1}\ \mbox{and}\ {A}_{m\times n} \ {S}_{n\times 1} = {B}_{m\times 1}\\
&\Rightarrow& {T}_{A}(Y) = {B}\ \mbox{and}\ {T}_{A}(S) = {B}\\
&\Rightarrow& {T}_{A}(Y)-{T}_{A}(S) = B - B = 0\\
&\Rightarrow& {T}_{A}(Y - S) = 0\\
&\Rightarrow& (Y - S) \in Ker({T}_{A})\\
&\Rightarrow& Y = U + S,\ \mbox{for some}\ U \in Ker({T}_{A}).
\end{eqnarray*}
Hence, $Y \in \ker ({T}_{A}) + S.$\\
Converse: Let us suppose that, $Y \in \ker ({T}_{A}) + S$. Then,
\begin{eqnarray*}
&& Y = U + S,\ \mbox{for some}\ U \in \ker ({T}_{A})\\
&\Rightarrow& {T}_{A}(Y) = {T}_{A}(U + S)\\
&\Rightarrow& {T}_{A}(Y) = {T}_{A}(U) + {T}_{A}(S)
\end{eqnarray*}
Since,  $U \in \ker ({T}_{A})$ and $S$ is a solution of the system ${A}_{m\times n} \ {X}_{n\times 1} = {B}_{m\times 1}$, therefore
\begin{eqnarray*}
&\Rightarrow& {T}_{A}(Y) = 0 + B\\
&\Rightarrow& {A}_{m\times n} \ {Y}_{n\times 1} = {B}_{m\times 1}.
\end{eqnarray*}
\end{proof}
The following corollary \ref{316} is an immediate consequence of \ref{315}.
\begin{corollary}\label{316}
Let $A= {^{1}A} e_1 + {^{2}A} e_2 =[\xi_{ij}]_{m\times n}  \in \C_2^{m\times n}, \ B = {^{1}B} e_1 + {^{2}B} e_2 = (b_{1}, b_{2}, \cdots, b_{m}) \in \C_{2}^{m}$  and $S = {^{1}S}e_1+{^{2}S}e_2 \in \C_{2}^{n}$ be a solution of the system ${A}_{m\times n} \ {X}_{n\times 1} = {B}_{m\times 1}$ of $m$ equations in $n$ variables $x_{1}, x_{2}, \cdots, x_{n} \in \C_{2}$, $i.e.$ $X = (x_{1}, x_{2}, \cdots, x_{n}) \in \C_{2}^{n}$. Then the set of solution of the system of equations ${A}_{m\times n} \ {X}_{n\times 1} = {B}_{m\times 1}$ is $Ker({T}_{A}) + S$.  
\end{corollary}
\begin{remark}\label{318}
If $A= {^{1}A} e_1 + {^{2}A} e_2 =[\xi_{ij}]_{m\times n}  \in \C_2^{m\times n}, \ B = {^{1}B} e_1 + {^{2}B} e_2 = (b_{1}, b_{2}, \cdots, b_{m}) \in \C_{2}^{m}$  and $S = {^{1}S}e_1+{^{2}S}e_2 \in \C_{2}^{n}$ is a solution of the system ${A}_{m\times n} \ {X}_{n\times 1} = {B}_{m\times 1}$ of $m$ equations in $n$ variables $x_{1}, x_{2}, \cdots, x_{n} \in \C_{2}$, $i.e.$ $X = (x_{1}, x_{2}, \cdots, x_{n}) \in \C_{2}^{n}$, then the map $f:\ker ({T}_{A})\to \ker ({T}_{A}) + S$, which is defined by the rule $f(U) = U + S$ for $U \in \ker ({T}_{A})$, is a one to one map. Thus, the sets $\ker ({T}_{A}) + S$ and $\ker ({T}_{A})$ are similar, $i.e.$ $\ker ({T}_{A}) + S \sim \ker ({T}_{A})$.
\end{remark}
\begin{theorem}
Let $A= {^{1}A} e_1 + {^{2}A} e_2 =[\xi_{ij}]_{m\times n}  \in \C_2^{m\times n}, \ (0, 0, \cdots, 0) \neq B = {^{1}B} e_1 + {^{2}B} e_2 = (b_{1}, b_{2}, \cdots, b_{m}) \in \C_{2}^{m}$  and ${A}_{m\times n} \ {X}_{n\times 1} = {B}_{m\times 1}$ be a system of $m$ equations in $n$ variables $x_{1}, x_{2}, \cdots, x_{n} \in \C_{2}$, $i.e.$ $X = (x_{1}, x_{2}, \cdots, x_{n}) \in \C_{2}^{n}$. Then we have the following:
\begin{enumerate}
\item If $\rho_{mr}(A) = 2n$ and $n < m$, then the system of equations ${A}_{m\times n} \ {X}_{n\times 1} = {B}_{m\times 1}$ has no solution or a unique solution.
\item If $\rho_{mr}(A) = 2m$ and $m < n$, then the system of equations ${A}_{m\times n} \ {X}_{n\times 1} = {B}_{m\times 1}$ has an infinite number of solutions.
\item If $\rho_{mr}(A) = 2m$ and $m = n$, then the system of equations ${A}_{m\times n} \ {X}_{n\times 1} = {B}_{m\times 1}$ has a unique solution.
\item If $\rho_{mr}(A) = k, k \neq 2n$ and $k \neq 2m$, then the system of equations ${A}_{m\times n} \ {X}_{n\times 1} = {B}_{m\times 1}$ has no solution or an infinite number of solutions.
\item If $\rho_{mr}(A) = 0$, then the system of equations ${A}_{m\times n} \ {X}_{n\times 1} = {B}_{m\times 1}$ has no solution.
\end{enumerate}
\end{theorem}
\begin{proof}
Suppose that $A= {^{1}A} e_1 + {^{2}A} e_2 =[\xi_{ij}]_{m\times n}  \in \C_2^{m\times n}, \ (0, 0, \cdots, 0) \neq B = {^{1}B} e_1 + {^{2}B} e_2 = (b_{1}, b_{2}, \cdots, b_{m}) \in \C_{2}^{m}$  and ${A}_{m\times n} \ {X}_{n\times 1} = {B}_{m\times 1} $ is a system of $m$ equations in $n$ variables $x_{1}, x_{2}, \cdots, x_{n} \in \C_{2}$, i.e. $X = (x_{1}, x_{2}, \cdots, x_{n}) \in \C_{2}^{n}$. Using \ref{32}, \ref{310}, \ref{316} and \ref{318}, we have
\begin{eqnarray*}
&& {T}_{A}:\C_2^n\to \C_2^m \ \mbox{such that}\ {T}_{A}(X) = {A}_{m\times n} \ {X}_{n\times 1}\ \mbox{and}\ 2n = \rho({T}_{A}) + \eta({T}_{A}),
\end{eqnarray*}
where $\eta({T}_{A})$ represents $\dim (\ker ({T}_{A}))$.
Therefore, $2n = \rho_{mr}(A) + \eta({T}_{A})$.
\begin{enumerate}
\item 
If $\rho_{mr}(A) = 2n$ and $n < m$, then the linear transformation ${T}_{A}:\C_2^n\to \C_2^m$ is not onto and $\eta({T}_{A}) = 0$. This implies 
\begin{eqnarray*}
	 \ker ({T}_{A}) = \left\lbrace (0, 0, \cdots, 0) \right\rbrace \subset\C_{2}^{n}.
\end{eqnarray*}
There are two cases:
\begin{enumerate}
\item 
If ${B} \in range\left({T}_{A} \right)$, then there exists $S = {^{1}S}e_1+{^{2}S}e_2 \in \C_{2}^{n}$ such that ${T}_{A}(S) = B$, $i.e.$ ${A}_{m\times n} \ {S}_{n\times 1} = {B}_{m\times 1}$ and $|\ker ({T}_{A}) + S| = 1$.
\item 
If ${B} \notin Range\left({T}_{A} \right)$, then ${T}_{A}(S) \neq B$ $i.e.$ ${A}_{m\times n} \ {S}_{n\times 1} \neq {B}_{m\times 1}$, for all $S = {^{1}S}e_1+{^{2}S}e_2 \in \C_{2}^{n}$.
\end{enumerate}
Therefore, the system of equations ${A}_{m\times n} \ {X}_{n\times 1} = {B}_{m\times 1}$ has no solution or a unique solution.
\item 
If $\rho_{mr}(A) = 2m$ and $m < n$, then the linear transformation ${T}_{A}:\C_2^n\to \C_2^m$ is onto and $\eta({T}_{A}) = 2n - 2m > 0$. This implies
$\ker ({T}_{A})$ has an infinite number of elements and ${B} \in Range\left({T}_{A} \right)$. Thus, there exists $S = {^{1}S}e_1+{^{2}S}e_2 \in \C_{2}^{n}$ such that ${T}_{A}(S) = B$, $i.e.$ ${A}_{m\times n} \ {S}_{n\times 1} = {B}_{m\times 1}$ and $\ker ({T}_{A}) + S$ has an infinite number of elements.
Therefore, the system of equations ${A}_{m\times n} \ {X}_{n\times 1} = {B}_{m\times 1}$ has an infinite number of solutions.
\item 
If $\rho_{mr}(A) = 2m$ and $m = n$, then the linear transformation ${T}_{A}:\C_2^n\to \C_2^m$ is onto and $\eta({T}_{A}) = 0$. This implies
\begin{eqnarray*}
 \ker ({T}_{A}) = \left\lbrace (0, 0, \cdots, 0) \right\rbrace \subset\C_{2}^{n}\ \mbox{and}\ {B} \in Range\left({T}_{A} \right).
\end{eqnarray*}
Thus, there exists $S = {^{1}S}e_1+{^{2}S}e_2 \in \C_{2}^{n}$ such that ${T}_{A}(S) = B$, $i.e.$ ${A}_{m\times n} \ {S}_{n\times 1} = {B}_{m\times 1}$ and $|\ker ({T}_{A}) + S| = 1$.
Therefore, the system of equations ${A}_{m\times n} \ {X}_{n\times 1} = {B}_{m\times 1}$ has a unique solution.
\item
If $\rho_{mr}(A) = k, k \neq 2n$ and $k \neq 2m$ then $k < 2n$ and $k < 2m$. Therefore, the linear transformation ${T}_{A}:\C_2^n\to \C_2^m$ is not onto and $\eta({T}_{A}) = 2n - k > 0$. This implies $\ker ({T}_{A})$ has an infinite number of elements.
There are two cases:
\begin{enumerate}
\item 
If ${B} \in range\left({T}_{A} \right)$, then there exists $S = {^{1}S}e_1+{^{2}S}e_2 \in \C_{2}^{n}$ such that ${T}_{A}(S) = B$, $i.e.$ ${A}_{m\times n} \ {S}_{n\times 1} = {B}_{m\times 1}$ and $Ker({T}_{A}) + S$ has an infinite number of elements.
\item 
If ${B} \notin range\left({T}_{A} \right)$, then ${T}_{A}(S) \neq B$ $i.e.$ ${A}_{m\times n} \ {S}_{n\times 1} \neq {B}_{m\times 1}$, for all $S = {^{1}S}e_1+{^{2}S}e_2 \in \C_{2}^{n}$.
\end{enumerate}
Therefore, the system of equations ${A}_{m\times n} \ {X}_{n\times 1} = {B}_{m\times 1}$ has no solution or an infinite number of solutions.
\item 
If $\rho_{mr}(A) = 0$, then $\eta({T}_{A}) = 2n$. This implies
\begin{eqnarray*}
&& \ker ({T}_{A}) = \C_{2}^{n}\ \mbox{and}\ range\left({T}_{A} \right) = \left\lbrace (0, 0, \cdots, 0) \right\rbrace \subset\C_{2}^{m}.\\
&\Rightarrow& {B} \notin range\left({T}_{A} \right),\ i.e.\ {A}_{m\times n} \ {S}_{n\times 1} \neq {B}_{m\times 1},\ \mbox{for all}\ S = {^{1}S}e_1+{^{2}S}e_2 \in \C_{2}^{n}.
\end{eqnarray*}
Therefore, the system of equations ${A}_{m\times n} \ {X}_{n\times 1} = {B}_{m\times 1}$ has no solution.
\end{enumerate}
\end{proof}
\bibliographystyle{unsrt}
\bibliography{references}
\end{document}